\documentclass[11pt,reqno]{amsart}
\usepackage[margin=3cm]{geometry}
\usepackage[utf8]{inputenc}
\usepackage[T1]{fontenc}
\usepackage{graphicx}
\usepackage{amsfonts,amsmath,amssymb}
\usepackage{mathrsfs}
\usepackage{amsthm}
\usepackage{thmtools}
\usepackage{mathtools, nccmath}
\usepackage{enumitem}
\usepackage[colorlinks, linkcolor=blue]{hyperref}
\usepackage{xcolor}

\usepackage{tikz}
\usepackage[
   backend=biber,        
   sorting=nyt,          
   citestyle=alphabetic, 
   bibstyle=alphabetic,  
   isbn=false,
   doi=false,
   url=false,
   maxalphanames=99,
   maxnames=99,
   giveninits=true,
   useprefix=true,
]{biblatex}
\newtheorem{theorem}{Theorem}
\newtheorem{lemma}[theorem]{Lemma}

\theoremstyle{definition}
\newtheorem{remark}[theorem]{Remark}
\newtheorem{definition}[theorem]{Definition}

\newcommand{\E}{\mathbf{E}}
\renewcommand{\P}{\mathbf{P}}

\newcommand{\R}{\mathbf{R}}

\newcommand{\Z}{\mathbf{Z}}
\newcommand{\N}{\mathbf{N}}

\renewcommand{\d}{\mathrm{d}}
\newcommand{\ds}{\displaystyle}
\newcommand{\ind}[1]{\mathbf{1}_{\{#1\}}}

\newcommand{\et}{\text{ and }}
\newcommand{\si}{\text{ if }}
\newcommand{\otw}{\text{ otherwise }}

\newcommand{\PP}{\mathcal{P}}

\author{Corentin Faipeur}
\address[C. Faipeur]{Laboratoire Analyse Géométrie Modélisation  (AGM), CY Cergy Paris Universit\'e, Cergy-Pontoise, France}
\email{corentin.faipeur@cyu.fr}

\title{Factor of i.i.d. with polynomial tails for the long-range Ising model}

\begin{document}
\begin{abstract}
    We prove that the very-high temperature long-range Ising model on $\Z^d$ is a factor of an i.i.d. process, and provide quantitative information on the coding radius. We do not assume that the model is ferromagnetic, but require the interactions to decay faster than $1/n^{2d}$.
    While previous constructions were known (in \cite{GalvesLocherbachOrlandi} under similar assumptions, or in \cite{HarelSpinka} in the entire subcritical phase for ferromagnetic interactions), the novelty of our approach is that we obtain an explicit polynomial upper bound on the tail of the coding radius. In particular, our construction yields a factor map with finite expected coding volume. This answers a question of a recent preprint by Chazottes, Gallo and Takahashi, see \cite[Question 2]{chazottes2026finitarycodinggaussianconcentration}.
\end{abstract}

\maketitle
\section{Introduction and main result}
\subsection{Long-range Ising model}
We consider the long-range Ising model on $\Z^d$, $d\geq 1$.
The configuration space is $\Omega= \{-1,+1\}^{\Z^d}$ and the interactions are ruled by \emph{coupling constants} $(J_{ij})_{i,j \in \Z^d}$ that satisfy $J_{ij}=J_{ji}$ and the following summability condition:
\begin{equation}\label{eq:summability}
    \forall i \in \Z^d, \qquad \sum_{j \in \Z^d} |J_{ij}| < \infty.
\end{equation}
Here, we will \emph{not} assume that the coupling constants are non-negative, so that the model is not necessarily ferromagnetic.
The Hamiltonian of the model is defined for every finite subset $\Lambda$ of $\Z^d$ ($\Lambda \Subset \Z^d$) by
\begin{equation*}
     \forall \sigma \in \Omega, \quad \mathcal H_{J,\Lambda}(\sigma):= - \sum_{\substack{\{i,j\} \cap \Lambda \neq \emptyset \\ i\neq j}} J_{ij} \sigma(i) \sigma(j).
\end{equation*}
The assumption \eqref{eq:summability} ensures that the above Hamiltonian is well defined. We can then consider Gibbs measures associated with a certain inverse temperature $\beta$. Since a sign can be included in the coupling constants, we can restrict ourselves to $\beta \geq 0$ without loss of generality.

\begin{definition}[Gibbs measure]
    For any $\Lambda \Subset \Z^d$, we define the (finite-volume) \emph{Gibbs measure} in $\Lambda$ at inverse temperature $\beta \geq0$ with boundary condition $\omega \in \Omega$ by
    \begin{equation*}
        \forall \sigma \in \Omega, \quad\mu_{\beta, J, \Lambda}^\omega(\sigma) :=\frac{1}{Z^\omega_{\beta, J, \Lambda}}\exp(-\beta \mathcal H_{J,\Lambda}(\sigma))\ind{\forall i \in \Lambda^c,\, \sigma(i)=\omega(i)}
    \end{equation*}
    where $Z^\omega_{\beta, J, \Lambda}$ is the partition function, that is the normalising constant that makes $\mu_{\beta, J, \Lambda}^\omega$ a probability measure.
    Then, we call an \emph{infinite-volume Gibbs measure} any distribution $\mu$ on $\Omega$ that satisfies the Dobrushin--Lanford--Ruelle equation, i.e.
    \begin{equation*}\label{eq:DLR}\tag{DLR}
        \forall \Lambda \Subset \Z^d, \quad\mu(\cdot)= \int_\Omega \mu^\omega_{\beta, J, \Lambda}(\cdot) \d \mu(\omega).
    \end{equation*}
\end{definition}
The above \eqref{eq:DLR} equation means that $\mu$ is an infinite-volume Gibbs measure for the model if and only if for all $\Lambda \Subset \Z^d$, for $\mu$-a.e. $\omega \in \Omega$,
$$\mu(\cdot \mid \forall i \in \Lambda^c,\sigma(i)=\omega(i) )=\mu^\omega_{\beta,J,\Lambda}(\cdot).$$
An important special case is when $\Lambda$ is a singleton:
for all $i \in \Z^d$, we have
$$\mu(\sigma(i)=-1 \mid \forall j \neq i, \sigma(j)=\omega(j))= \dfrac{1}{1+\exp(2\beta S_i(\omega))}$$
and
$$\mu(\sigma(i)=+1 \mid \forall j \neq i, \sigma(j)=\omega(j))= \dfrac{1}{1+\exp(-2\beta S_i(\omega))},$$
where $S_i(\omega):=\sum_{j \neq i} J_{ij} \omega(j)$.

\medskip
We will focus on \emph{translation-invariant} coupling constants, which means that they are of the form $J_{ij} = J(j-i)$ for some even function $J:\Z^d \setminus \{0\} \to \R$.
A classical choice in the literature for the coupling constants is to consider $J(x)=|x|^{-\alpha}$ for some decay exponent $\alpha>d$ (this last condition being equivalent to \eqref{eq:summability}), or say $J(x)=O(|x|^{-\alpha})$ to include non-ferromagnetic models.
Thus, the exponent $\alpha$ is a parameter that controls the decay of the interactions with the distance, and that influences the \emph{phase transition} of the model.

A central question in statistical mechanics is to determine whether there is a unique or several measures satisfying the \eqref{eq:DLR} equation (existence of at least one infinite-volume Gibbs measure is always guaranteed, by compactness arguments).
General criteria, e.g. by Dobrushin \cite{Dobrushin} or van den Berg--Maes \cite{VdBMaes}, ensure that uniqueness holds for $\beta$ small enough.
The first result concerning phase transition for the long-range Ising model was obtained in 1966 by Ginibre, Grossman and Ruelle \cite{GinibreGrossmanRuelle}. They generalise the Peierls argument for the nearest-neighbour Ising model to prove that the long-range model admits several Gibbs measures at large enough $\beta$ in dimension $d\geq 2$ when $\alpha>d+1$.
In the ferromagnetic case, i.e. $J \geq0$, the Griffiths inequalities \cite{Griffiths} (or GKS inequalities \cite{KellySherman}) enable one to prove the coexistence of Gibbs measures for large $\beta$ in the full regime $\alpha>d$.

The one-dimensional case offers an interesting different behaviour: uniqueness of the Gibbs measure has been proved at any inverse temperature in \cite{Ruelle} when $\alpha>2$, ruling out the presence of a phase transition; Kac and Thompson conjectured in \cite{KacThompson} that the model undergoes a phase transition for $\alpha \in (1,2]$; this has been verified by Dyson \cite{Dyson} for $\alpha \in(1,2)$ and then by Fr\"ohlich and Spencer \cite{FrohlichSpencer} in the case $\alpha=2$.

In the present article, we will restrict ourselves to a ``very high temperature'' regime, that is, a perturbative regime of small enough inverse temperature $\beta$. In particular, for any choice of $d$ and $\alpha$ (or $J$), $\beta$ is taken sufficiently small so that uniqueness of the infinite-volume Gibbs measure is achieved.
We denote by $\mu_{\beta, J}$ this unique measure. As translations of $\mu_{\beta, J}$ would also be Gibbs measures, it is clear by uniqueness that $\mu_{\beta, J}$ is translation-invariant.

The purpose of this paper is to show that for $\beta$ small enough, this measure $\mu_{\beta, J}$ can be written as an explicit factor of an i.i.d. field, with a polynomially decaying coding radius.
Prior to stating the result precisely, we provide the necessary background on factor maps and coding radius.

\subsection{Factor maps and coding radius}\label{sec:factormaps}
There is an extensive literature about factors of i.i.d. for random fields, see e.g. \cite{VdB_steif}, \cite{HaggstromSteif}, \cite{Spinka_Ising}, \cite{Spinka_mixing}, \cite{HarelSpinka}, \cite{Faipeur}.
We will not provide a full overview of the topic here, and instead refer to the above references for background and more detailed insights.
Below we give the relevant definitions for our purpose.

Let $\mu$ be a probability measure on $\Omega=\{\pm1\}^{\Z^d}$.
We say that $\mu$ is a \emph{factor of an i.i.d. field} (abbreviated FIID in the following) if for some Borel space $(A,\mathcal{A})$, there exist a measurable map $F:A^{\Z^d} \to \Omega$ which commutes with every translation of $\Z^d$ and a collection $Y=(Y_i)_{i \in \Z^d}$ of i.i.d. (independent and identically distributed) $A$-valued random variables such that $F(Y)$ has distribution $\mu$.
In particular, $\mu$ is necessarily translation-invariant.
The factor map $F$ is called a \emph{coding} for $\mu$.
An important object in the study of codings is the so-called \emph{coding radius}. It quantifies how far one needs to look in the i.i.d. process to determine the value of the random field at some site. For all $y \in A^{\Z^d}$, let
$$r_0(y):= \inf\{r \geq 0 \mid \forall y' \in A^{\Z^d}, y'_{\mid B_r} = y_{\mid B_r} \implies F(y')_0=F(y)_0\}$$
where $B_r$ denotes the ball of radius $r$ around $0$ in $\Z^d$, and $y_{\mid B_r}$ the restriction of $y$ to $B_r$.
The coding radius of the factor map $F$ is the random variable
$$R:=r_0(Y).$$

When $R$ is almost surely finite, we say that $\mu$ is a finitary factor of an i.i.d. field, abbreviated FFIID.
Moreover, we say that $\mu$ is an \emph{FFIID with exponential tails} if the coding radius satisfies for some $c,C>0$,
$$\forall n \geq 1, \qquad \P(R>n) \leq Ce^{-cn};$$
we say that $\mu$ is an \emph{FFIID with polynomial tails} if the coding radius satisfies for some $\rho,C>0$,
$$\forall n \geq 1, \qquad \P(R>n) \leq Cn^{-\rho}.$$

A related notion is the \emph{coding volume}, which is the random variable $|B_R|$; it is therefore of order $R^d$. When looking for finitary coding of random fields, it is valuable to obtain a coding with finite expected coding volume, as it means that the expected number of random variables in the i.i.d. process needed to compute the field at the origin is finite.
For example, this notion allows to distinguish the critical regime from the subcritical and supercritical ones for the nearest-neighbour ferromagnetic Ising model: it is proved in \cite{VdB_steif} that the model is FFIID with exponential tails in the subcritical phase, but that at the critical point, any finitary coding must have infinite expected coding volume; the existence of a finitary coding at the critical point follows from the continuity of the phase transition; another major result of \cite{VdB_steif} is that no finitary coding exists in the supercritical phase.

We end this section by noticing that the tail of the coding radius always controls the \emph{two-point function} of the model.
Let us clarify this for the subcritical Ising model.
We denote by $\langle \cdot\rangle_{\beta,J}$ the expectation with respect to $\mu_{\beta,J}$.
The two-point function is defined by $\langle \sigma_0\sigma_x\rangle_{\beta,J}$, for $x \in \Z^d$; since $\langle \sigma_x\rangle_{\beta,J}=0$ for all $x$, this corresponds to the covariance between two spins under $\mu_{\beta,J}$. It is an important goal in statistical physics to understand the decay of this quantity as $|x| \to \infty$. 
For any coding of the model, whose coding radius is denoted by $R$, the following inequality holds:\footnote{Suppose that $F$ is a coding for $\mu_{\beta,J}$ from the i.i.d. field $Y$, with coding radius $R$; more precisely, let $R_0:=r_0(Y)$ and $R_x:=r_x(Y)$, $r_x$ being defined similarly to $r_0$, but taking $x$ as origin.
Take $r= \lceil|x|/2 \rceil$, so that the balls $B_{r-1}$ and $x+B_{r-1}$ are disjoint.
Write $p=\mathbf P(R\ge r)$ and decompose $\langle\sigma_0\sigma_x\rangle_{\beta,J}=\E[F(Y)_0F(Y)_x]$~as
\begin{equation*}
\E\big[F(Y)_0F(Y)_x\,\mathbf 1_{\{R_0<r\}}\mathbf 1_{\{R_x<r\}}\big]
+ {\E\big[F(Y)_0F(Y)_x \,\mathbf 1_{\{R_0\geq r\}\cup \{R_x\geq r\}}\big] =:(\mathrm I)+(\mathrm{II})}.
\end{equation*}
For the first term, $F(Y)_0\mathbf 1_{\{R_0<r\}}$ and $F(Y)_x\mathbf 1_{\{R_x<r\}}$ have the same distribution, and they are $\sigma(Y_{\mid B_{r-1}})$- and $\sigma(Y_{\mid x+B_{r-1}})$-measurable respectively, hence independent.
Then, as $\E[F(Y)_0]=0$, we have $\E[F(Y)_0\mathbf 1_{\{R_0<r\}}] = -\E[F(Y)_0\mathbf 1_{\{R_0\geq r\}}]$ which is, in absolute value, less than $p$. It gives $|(\mathrm I)|\le p^2$.
For the second term, simply using $|F(Y)_0F(Y)_x| =1$ yields $|\mathrm{(II)}|\leq \P(\{R_0\geq r\}\cup \{R_x\geq r\})=2p-p^2$, because the two events are independent.
Combining these two upper bounds gives $|\langle\sigma_0\sigma_x\rangle_{\beta,J}| \leq p^2+(2p-p^2)=2p$.}

\begin{equation}\label{eq:two_point_fct_and_radius}
    \forall x \in \Z^d, \qquad|\langle\sigma_0\sigma_x\rangle_{\beta,J}| \leq 2 \P(R\geq |x|/2).
\end{equation}

\subsection{Main result and related works}
Fix a function $J:\Z^d \setminus\{0\} \to \R$ encoding the coupling constants. Our main assumption on $J$ is that
\begin{equation}\label{eq:form_J}
    \forall x \neq 0, \qquad |J(x)| \leq |x|^{-\alpha}
\end{equation}
for some decay exponent $\alpha>d$
(N.B. we lose no generality by not including a multiplicative constant in the upper bound, since this constant can be absorbed in $\beta$).
We again stress that this induces no restriction on the sign of the interactions.
Recall that when $\beta$ is small enough, there is a unique infinite-volume Gibbs measure, denoted $\mu_{\beta, J}$, satisfying the \eqref{eq:DLR} equation associated with the Hamiltonian defined by $J$. 
This is in particular the case in the following statement, which is the main result of this article.

\begin{theorem}\label{thm:main}
    For $\alpha>2d$, there exists $\beta_0:=\beta_0(\alpha,d) \in(0,\infty)$ such that for every $\beta<\beta_0$, $\mu_{\beta,J}$ is an FFIID with polynomial tails. More precisely, there is a coding for $\mu_{\beta, J}$ whose coding radius $R$ satisfies
    \begin{equation*}
        \forall n \geq1, \quad \P(R >n) \leq Cn^{d-\alpha}
    \end{equation*}
    for some constant $C:= C(\alpha,\beta,d) \in (0,\infty)$.
\end{theorem}

\begin{remark}
    Let us make a few comments about this result.
    \begin{itemize}
        \item This statement implies that the coding volume $|B_R|$ satisfies
        \begin{equation*}
            \forall n \geq1, \quad \P(|B_R| >n) \leq \hat Cn^{1-\alpha/d}
        \end{equation*}
        for some constant $\hat C:=\hat C(\alpha,\beta,d) \in (0,\infty)$.
        In particular, since $\alpha>2d$, our coding of $\mu_{\beta, J}$ has finite expected coding volume.
        \item The assumption $\alpha>2d$ is necessary, not only to get a finite expected coding volume; in particular, we cannot obtain a finitary coding with our method in the case $\alpha \in (d,2d]$.
        \item If we replace assumption \eqref{eq:form_J} by $|J(x)| \leq |x|^{-\alpha} L(|x|)$ for some slowly varying function $L$ (see Section \ref{sec:prelim} for the definition), we obtain the same result with the function $L$ appearing on the RHS, that is
        $$\forall n \geq1, \quad \P(R >n) \leq Cn^{d-\alpha} L(n)$$
        for some constant $C>0$. Interestingly, this allows to get the result for $\alpha=2d$ and e.g. $L(n)=\ln(n)^{-(1+\varepsilon)}$ for some $\varepsilon>0$.
        We refer to Remark \ref{rem:hyp_J} at the end of the paper for more details.
        \item For concreteness, we choose to consider the model on $\Z^d$, but nothing in the proof strongly relies on the structure of the lattice. One could do the same proof on any infinite transitive graph and obtain an analogous result.
    \end{itemize}
\end{remark}

As already mentioned in the previous section, the first results regarding finitary codings for the Ising model were obtained by van den Berg and Steif \cite{VdB_steif}.
They showed that when several infinite-volume Gibbs measures coincide, finitary codings cannot exist.
They also proved that when the interactions are finite-range and ferromagnetic (i.e. $J$ has finite support and $J\geq0$), the subcritical Ising model is an FFIID with exponential tails.
At the same time, Häggström and Steif \cite{HaggstromSteif} proved the same result in the finite-range case at very high temperature (i.e. $\beta$ small enough so that a ``high-noise'' condition is fulfilled), but without assuming $J$ positive.
The methods of these two papers cannot apply when the interactions have infinite range.

For the long-range Ising model, Galves, L\"ocherbach, and Orlandi \cite{GalvesLocherbachOrlandi} proved the existence of a finitary coding at very high temperature (with a similar ``high-noise'' condition to \cite{HaggstromSteif}),
not necessarily in the ferromagnetic case, with a stronger condition than \eqref{eq:summability} on $J$; they consider interactions of more than two spins, but for pairwise interactions, their condition is equivalent to our $\alpha>2d$ (see Remark 2 in that paper, and compare it with our \eqref{eq:integrability_condition}).
In fact, their goal was to construct a \emph{perfect sampling algorithm} rather than a factor map, so they were mostly interested in controlling the stopping time of the algorithm. Thus, we cannot deduce quantitative information on the tail of the coding radius from their result.
Then, Harel and Spinka have constructed in \cite{HarelSpinka} a finitary coding in the entire regime of uniqueness of the Gibbs measure, in the ferromagnetic case. Their techniques rely on monotonicity, but are not restricted to a perturbative regime of temperature.
Their coding satisfies the following inequality:
$$\P(R>n) \leq d_{\operatorname{TV}}(\mu^+_{\beta, J, B_n} (\sigma_0 \in \cdot);\mu^-_{\beta, J, B_n} (\sigma_0 \in \cdot))=\langle \sigma_0\rangle^+_{\beta,J,B_n},$$
where $d_{\operatorname{TV}}$ denotes the total variation distance (the inequality is Theorem 7 in \cite{HarelSpinka} and the equality is a simple computation).
The tail of the coding radius is thus upper bounded by the so-called \emph{magnetisation} of the model. However, it is unclear how this function decays as $n$ goes to $\infty$.
The GKS inequalities easily imply that $\langle \sigma_0\rangle^+_{\beta,J,B_n}$ is lower bounded by $c \beta n^{d-\alpha}$ for some $c>0$ if $J(x) \sim |x|^{-\alpha}$, but we are unable to prove an upper bound of the same order.
\medskip

Theorem \ref{thm:main} therefore provides, to the best of our knowledge, the first quantitative polynomial upper bound on the coding radius tails for the long-range Ising model.
We further notice that a coding for the long-range Ising model cannot have faster-than-polynomial tails.
We have already observed in Section \ref{sec:factormaps} that the tail of the coding radius is lower bounded by the two-point function, see \eqref{eq:two_point_fct_and_radius}. The fact that the two-point function cannot decay faster than the potential (i.e. faster than $J$) goes back to Iagolnitzer and Souillard \cite{IagolnitzerSouillard}.
Newman and Spohn further proved in an unpublished manuscript that $\langle\sigma_0 \sigma_x\rangle_{\beta,J} \sim c(\beta)J(x)$ as $|x|\to \infty$ (see \cite{Aoun} for a published paper, which extends the result to the random-cluster and Potts models). These two results hold only in the ferromagnetic case. Thus, say for $J(x) =|x|^{-\alpha}$ and $\beta$ small enough, we have the following chain of inequalities:
$$cn^{-\alpha} \leq \P(R>n) \leq Cn^{d- \alpha}$$
for some constants $c,C \in (0,\infty)$.

This answers a question recently raised by Chazottes, Gallo, and Takahashi \cite[Question 2]{chazottes2026finitarycodinggaussianconcentration}.
They were looking for examples of Gibbs measures which are FFIID with polynomial tails, still with integrable (or square-integrable) coding volume.\footnote{Here, in the sense that any coding of the model must have at least polynomial tails (otherwise, one has plenty examples of finitary codings with exponential tails, thus with finite moments for the coding volume).} The very high temperature long-range Ising models with decay exponent $\alpha>2d$ therefore provide a class of examples.
Furthermore, one can deduce from their work and from Theorem \ref{thm:main} that in this case, $\mu_{\beta,J}$ satisfies Gaussian concentration (apply \cite[Theorem 3.3]{chazottes2026finitarycodinggaussianconcentration} to the coding provided by our Theorem \ref{thm:main}, noticing that it satisfies the short-range factorisation property \cite[Definition 3.2]{chazottes2026finitarycodinggaussianconcentration} with constant $1/2$).

\section{Construction of the finitary coding}
\subsection{Preliminaries: balls of $\Z^d$, regularly varying functions and stochastic domination.}\label{sec:prelim}
We denote by $\N=\{0,1,2,\dots\}$ the set of non-negative integers.
For all $r\in \N$, let 
$$B_r:=\{x \in \Z^d, |x|\leq r\}$$
be the ball of radius $r$ centered at the origin in the $d$-dimensional lattice ($|x|$ denotes the $1$-norm of $x\in \Z^d$).
Let $\partial B_{r+1}:= B_{r+1} \setminus B_r$ be the set of sites in $\Z^d$ at distance exactly $r+1$ from the origin, and let $B_r^* := B_r \setminus \{0\}$ be the ball of radius $r$ excluding the origin.
We shall use the following basic estimates on the asymptotic size of balls and spheres in $\Z^d$: there are constants $v_d$ and $\kappa_d$ in $(0,\infty)$ such that
\begin{equation}\label{eq:asymptotic_size}
    |B_r| \sim |B_r^*| \sim v_d r^d \quad \et \quad |\partial B_r| \sim \kappa_d r^{d-1} \quad \text{as }r\to \infty.
\end{equation}

We then recall some basic definitions of the theory of \emph{regularly varying functions}.
A (Lebesgue measurable) function $L: \R \to (0, \infty)$ is said to be slowly varying (at $+\infty$) if
$$\forall a>0, \quad \frac{L(at)}{L(t)} \underset{t\to +\infty}{\longrightarrow}1.$$
Trivial examples of slowly varying functions are the constant functions.
A (Lebesgue measurable) function $F:\R \to (0,\infty)$ is said to be regularly varying (at $+\infty$) with index $\rho \in \R$ if
$$\forall a>0, \quad \frac{F(at)}{F(t)} \underset{t\to +\infty}{\longrightarrow}a^\rho$$
or, equivalently, if
$$\forall t, \quad F(t)=t^\rho L(t)$$
for some slowly varying function $L$.
In particular, a probability measure $\mu$ on $\N$ is said to have \emph{regularly varying tails} with index $\rho$ if
$$\forall n \geq1, \quad \mu([n,\infty))=n^\rho L(n)$$
for some slowly varying function $L$.
The final argument of the proof makes use of a theorem on random walks whose jump distribution has regularly varying tails.

Finally, for any two probability measures $\mu$ and $\mu'$ on $\N$, we say that $\mu$ is \emph{stochastically dominated} by $\mu'$ if
for every non-decreasing bounded function $f:\N \to \R$, one has $\int f \d \mu \leq \int f \d \mu'$.
It is well-known that this is equivalent to
$$\forall n \geq 1, \quad\mu([n, \infty)) \leq \mu'([n,\infty)),$$
and also to the existence of a coupling $(X,X')$ of $\mu$ and $\mu'$ such that $X \leq X'$ almost surely.

\subsection{Glauber dynamics and coupling-from-the-past}
As with all the previous constructions of finitary codings mentioned above, ours is based on the coupling-from-the-past algorithm, invented by Propp and Wilson in their seminal work \cite{propp_wilson}.
The idea is to couple Markov processes whose invariant law is the targeted measure, initialised from any possible initial configuration at a large negative time. If all the processes coincide at a certain site at time 0, then the spin at this site has the correct distribution.

The Glauber dynamics, originally introduced in \cite{Glauber} by Roy J. Glauber, is a continuous-time Markov process on $\Omega$ whose invariant distribution is $\mu_{\beta, J}$.
The graphical representation of the dynamics, that we describe below and which goes back to Harris \cite{Harris}, provides a coupling of processes starting from any time and any initial configuration, which is well-suited for implementing coupling-from-the-past constructions.
The idea is to attach i.i.d. ``exponential clocks'' to every vertex of $\Z^d$ and to update the spin at a vertex when its clock rings, according to the conditional distribution given the rest of the configuration.
This procedure, also called the heat-bath algorithm, clearly leaves invariant $\mu_{\beta, J}$.

Let $\PP$ be a Poisson point process (PPP) of intensity $\# \otimes \d t \otimes \d u$, where $\#$, $\d t$, $\d u$ denote the counting measure on $\Z^d$, the Lebesgue measure on $\R$ and the Lebesgue measure on $[0,1]$, respectively. Points of $\PP$ are called the \emph{update marks} of the dynamics.
The law of $\PP$ is a product measure over $\Z^d$; this is the product measure from which we will construct a factor map to $\mu_{\beta, J}$.
For $(i,t,u) \in \PP$, the three coordinates of the update mark should be interpreted as follows: $i$ is the site carrying the spin that is updated, $t$ is the time at which the update occurs and $u$ is an extra-randomness from which we sample the new spin at $i$.
Define the \emph{update function} $\varphi:\R \times[0,1] \to \{\pm1\}$ by
\begin{equation}\label{eq:phi_Ising}
    \varphi(s,u) =\begin{cases}
        -1& \si u\leq (1+e^{2\beta s})^{-1} \\
        +1& \otw \end{cases}.
\end{equation}
Note for future purpose that $\varphi$ is non-decreasing in both variables.
For every update mark $(i,t,u)\in \PP$, if the process is in configuration $\sigma$ before the update, then the new value of the spin at $i$ is given by $\varphi(\sum_{j \ne i} J_{ij} \sigma(j), u)$; the other spins remain unchanged.
For any $\tau \in \R$ and $\omega \in \Omega$, denote by $(\sigma^{\omega,\tau}_t)_{t \geq \tau}$ the process starting at time $\tau$ with initial configuration $\omega$, and whose evolution is governed by the update marks $\PP$.
It means that we have $\sigma^{\omega,\tau}_\tau=\omega$ and for every $(i,t,u) \in \PP$ with $t> \tau$,
\begin{equation}\label{eq:coupled_processes}
    \forall j\neq i, \sigma^{\omega,\tau}_t(j)=\sigma^{\omega,\tau}_{t-}(j) \et \sigma^{\omega,\tau}_t(i)=\varphi\left(\sum_{j \neq i} J_{ij}\sigma^{\omega,\tau}_{t-}(j), u\right).
\end{equation}

Now, let us see how coupling-from-the-past arguments enable one to obtain an exact sampling of $\mu_{\beta, J}$ as a factor of $\PP$.
For every $i \in \Z^d$, let 
$$\tau_i=\sup\{\tau \leq0 \mid \forall \omega, \omega' \in \Omega, \ \sigma_0^{\omega, \tau}(i)=\sigma^{\omega',\tau}_0(i)\},$$
that is, $\tau_i$ is the closest to 0 negative time such that for processes starting at time $\tau_i$, the spin at site $i$ at time $0$ does not depend on the initial configuration.
Observe that $\tau_i$ is a random time, measurable with respect to $\PP$.
We can then define a random field $\sigma^*=(\sigma^*(i))_{i \in \Z^d}$ as a measurable function of $\PP$ (so an FIID), by
$$\forall i \in \Z^d, \quad\sigma^*(i):=\sigma_0^{+, \tau_i}(i)$$
(here we arbitrarily put $\omega \equiv+1$ as initial configuration, as by definition of $\tau_i$, it does not influence the value of $\sigma^*(i)$).
Then,\textbf{ assuming that the $\tau_i$ are a.s. finite}, the fact that $\sigma^*$ is distributed according to $\mu_{\beta, J}$ is a classical infinite-volume adaptation of the coupling-from-the-past argument of \cite{propp_wilson}, which already appears in many references cited above (see \cite[Lemma 3.6]{VdB_steif}, \cite[Theorem 2.2]{HaggstromSteif}, \cite[Proposition 2.3]{Spinka_Ising}, \cite[Proposition 2.1]{Faipeur}).
This therefore produces a coding for $\mu_{\beta, J}$.
It remains to prove that the $\tau_i$ are a.s. finite and to control the tail of the coding radius.

In Section \ref{sec:Backward exploration}, we outline a proof, in the same spirit as the one in \cite{HaggstromSteif}, which shows that finite-range Ising models (i.e. when $J$ has a finite support) are FFIID with exponential tails for $\beta$ sufficiently small. We then adapt this proof to our long-range case in Section \ref{sec:Proof of Theorem}, proving Theorem \ref{thm:main}.

\subsection{Backward exploration}\label{sec:Backward exploration}
In this section, we explain the strategy to prove that the $\tau_i$ are a.s. finite for finite-range models.
Since they all have the same law (while not being independent), it is enough to prove that $\tau_0$ is a.s. finite.

We explore the dynamics backwards, starting from the space-time position $(0,0)$. We go back to the last update of site $0$ before time $0$, that is, we consider the update mark $$(0,t_0,u_0) \in \PP \text{ with }t_0=\sup\{t<0 \mid \exists u_0 \in[0,1], (0,t,u_0)\in \PP\}$$ ($t_0$ is a.s. finite by properties of the PPP).
Let $S_0:= \sum_{x \neq 0} |J(x)|$ (which is finite by assumption); when we apply the update function in the dynamics, the first variable takes values between $-S_0$ and $S_0$.
By monotonicity of the update function, it implies that if $u_0 \leq (1+e^{2\beta S_0})^{-1}$, then the new spin at site $0$ produced by the update mark is $-1$ whatever the configuration of the process at time $t_0-$; in particular, since the spin at $0$ is no longer changed up to time $0$, we have $\sigma^{\omega, \tau}_0(0)=-1$ for any $\tau<t_0$ and any initial configuration $\omega$; a similar situation happens if $u_0> (1+e^{-2\beta S_0})^{-1}$, where this time we know that there is a spin $+1$ at site $0$.
When $u_0$ is between $(1+e^{2\beta S_0})^{-1}$ and $(1+e^{-2\beta S_0})^{-1}$, the value of the spin produced by the update does depend on the configuration at time $t_0-$.

In the finite-range case, say for the nearest-neighbour Ising model, we do not need to reveal the entire configuration but only the neighbours of $0$. Thus, we pursue the exploration from the space-time positions corresponding to time $t_0$ and neighbouring sites of the origin. If for all of them, the next update mark we find has a third coordinate $u$ which is not in the uncertainty interval $\big](1+e^{2\beta S_0})^{-1},(1+e^{-2\beta S_0})^{-1}\big]$, then we can stop the exploration; indeed, in this case, all the spins of the neighbours at time $t_0-$ are determined independently of the initial configuration, so is the spin at $0$ after the update of time $t_0$, even though $u_0$ was in the uncertainty interval. Otherwise, we continue again the exploration etc.

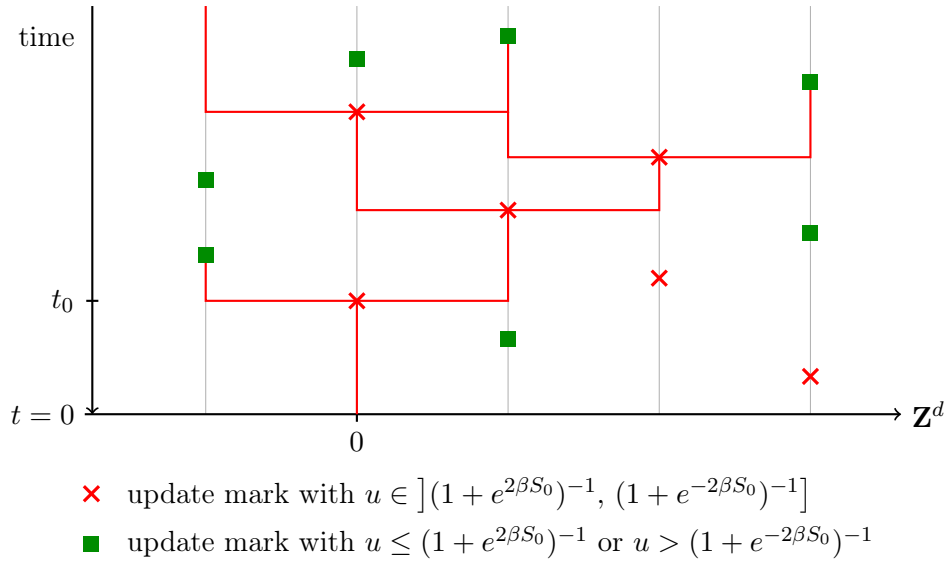
\begin{figure}[h]
\centering
\begin{tikzpicture}
  \def\ymax{5.4}

  \foreach \x in {-3,-1,1,3,5} {
    \draw[gray!60, thin] (\x,0) -- (\x,\ymax);
  }

  \draw[->, thick] (-4.6,0) -- (6.2,0) node[right] {$\Z^d$};
  \draw[thick] (-1,0) -- (-1,-0.1);
  \node[below] at (-1,-0.1) {$0$};

  \draw[<-, thick] (-4.5,0) -- (-4.5,\ymax); 
  \node[left] at (-4.6,5) {time};
  \node[left] at (-4.6,0) {$t=0$};
  \draw[thick] (-4.58,1.5) -- (-4.42,1.5);
  \node[left] at (-4.6,1.5) {$t_0$};

  \draw[red, thick] (-1,0) -- (-1,1.5) -- (-3,1.5) -- (-3,2.1);
  \draw[red, thick] (-1,1.5) -- (1,1.5) -- (1,2.7) -- (-1,2.7) -- (-1,4) -- (-3,4) -- (-3, \ymax);
  \draw[red, thick] (-1,4) -- (1,4) -- (1,5);
  \draw[red, thick] (1,2.7) -- (3,2.7) -- (3,3.4) -- (1,3.4) -- (1,4);
  \draw[red, thick] (3,3.4) -- (5,3.4) -- (5, 4.4);
  
  \draw[red, very thick] (-1.1,1.4) -- (-0.9,1.6);\draw[red, very thick] (-1.1,1.6) -- (-0.9,1.4);
  \draw[red, very thick] (0.9,2.6) -- (1.1,2.8);\draw[red, very thick] (0.9,2.8) -- (1.1,2.6);
  \draw[red, very thick] (-1.1,3.9) -- (-0.9,4.1);\draw[red, very thick] (-1.1,4.1) -- (-0.9,3.9);
  \draw[red, very thick] (2.9,1.7) -- (3.1,1.9);\draw[red, very thick] (2.9,1.9) -- (3.1,1.7);
  \draw[red, very thick] (2.9,3.3) -- (3.1,3.5);\draw[red, very thick] (2.9,3.5) -- (3.1,3.3);
  \draw[red, very thick] (4.9,0.4) -- (5.1,0.6);\draw[red, very thick] (4.9,0.6) -- (5.1,0.4);
  
  \fill[green!55!black] (-3.1,2) rectangle (-2.9,2.2);
  \fill[green!55!black] (-3.1,3) rectangle (-2.9,3.2);
  \fill[green!55!black] (0.9,4.9) rectangle (1.1,5.1);
  \fill[green!55!black] (4.9,2.3) rectangle (5.1,2.5);
  \fill[green!55!black] (0.9,0.9) rectangle (1.1,1.1);
  \fill[green!55!black] (-1.1,4.6) rectangle (-0.9,4.8);
  \fill[green!55!black] (4.9,4.3) rectangle (5.1,4.5);
\end{tikzpicture}

\begin{center}
\renewcommand{\arraystretch}{1.4}
\begin{tabular}{cl}
  \tikz{\draw[red, very thick] (-0.1,-0.1) -- (0.1,0.1);
        \draw[red, very thick] (-0.1,0.1) -- (0.1,-0.1);}
  & update mark with $u \in \big](1+e^{2\beta S_0})^{-1},\,(1+e^{-2\beta S_0})^{-1}\big]$ \\
  \tikz{\fill[green!55!black] (-0.1,-0.1) rectangle (0.1,0.1);}
  & update mark with $u \le (1+e^{2\beta S_0})^{-1}$ or $u > (1+e^{-2\beta S_0})^{-1}$ \\
\end{tabular}
\end{center}
\caption{Illustration of a backward exploration of the dynamics (in dimension $d=1$). Note that some update marks are not revealed by the exploration, and also that two exploration paths (in red) can merge, so that the resulting graph is not necessarily a tree.}
\end{figure}

Let 
\begin{equation*}
    \gamma(\beta):=(1+e^{2\beta S_0})^{-1}+1-(1+e^{-2\beta S_0})^{-1}=1-\tanh(\beta S_0)
\end{equation*}
be the proportion of update marks that allow us to stop the exploration; this quantity corresponds to the multigamma admissibility of \cite{HaggstromSteif}.
The backwards exploration is dominated by a Galton--Watson process of reproduction law
$$\pi=\gamma(\beta) \delta_0 + (1-\gamma(\beta)) \delta_{2d}$$
(the ``$2d$'', i.e. the number of neighbours of $0$, being replaced by the cardinal of the support of $J$ for other finite-range cases).
More precisely, the backwards exploration stops if the Galton--Watson process dies out, and the number of update marks revealed during the exploration is less than the total population.
If the expectation of $\pi$ is less than $1$, which is the case when $\beta$ is sufficiently small, the Galton--Watson process is subcritical: it dies out almost surely and the total population has exponential tails.
In this case, it immediately implies that $\tau_0$ is a.s. finite and that the coding volume $|B_R|$ has exponential tails.
The fact that the expectation of $\pi$ is less than 1 is equivalent to $\gamma(\beta) >1-1/(2d)$, which is precisely the high-noise condition of \cite{HaggstromSteif}.

If one tries to apply this method for the long-range case, one ends up with a Galton--Watson process of reproduction law
$$``\pi=\gamma(\beta) \delta_0 + (1-\gamma(\beta)) \delta_{\infty} ".$$
Our strategy, implemented in the next section, will be to reveal the configuration only in a finite ball around the updated site, the smallest that enables one to determine the value of the new spin.

\subsection{Proof of Theorem \ref{thm:main}}\label{sec:Proof of Theorem}

\subsubsection{Glauber dynamics with a new update function}
For all $r \in \N=\{0,1,2,\dots\}$ and $\sigma \in \Omega$, let 
$$p^-_r(\sigma)=\left(1+e^{2 \beta \sum_{|x|\leq r} J(x) \sigma(x) +2 \beta \sum_{|x|> r} |J(x)|}\right)^{-1}$$
and 
$$p^+_r(\sigma)=\left(1+e^{-2 \beta\sum_{|x|\leq r} J(x) \sigma(x) +2 \beta \sum_{|x|> r} |J(x)|}\right)^{-1};$$
$p_r^-(\sigma)$ (resp. $p_r^+(\sigma)$) corresponds to the minimal probability under $\mu_{\beta, J}$ for the spin at the origin to equal $-1$ (resp. $+1$) given that the rest of the configuration coincide with $\sigma$ on $B^*_r$ (in both cases, the minimum is taken over configurations $\sigma'$ that agrees with $\sigma$ on $B^*_r$; it is attained in the first case by taking $\sigma'(x)= \frac{J(x)}{|J(x)|}$ for all $x\notin B_r$, and the opposite in the second case).
Note that $p_0^-(\sigma)$ and $p_0^+(\sigma)$ do not depend on $\sigma$, and in fact correspond to the $(1+e^{2 \beta S_0})^{-1}$ and $(1+e^{-2 \beta S_0})^{-1}$ of the previous section.
The two sequences $(p^-_r(\sigma))_{r \in \N}$ and $(p^+_r(\sigma))_{r \in \N}$ are non-decreasing and converge respectively to $(1+e^{2\beta S_0(\sigma)})^{-1}$ and $(1+e^{-2\beta S_0(\sigma)})^{-1}$ (where $S_0(\sigma)=\sum_{x \neq 0} J(x) \sigma(x)$), that is, the conditional probability that the spin at the origin is $-1$ (resp. $+1$) given that the rest of the configuration coincide with $\sigma$.

Then, let $$\gamma_r(\sigma):=p_r^-(\sigma)+p_r^+(\sigma) \quad \et \quad\gamma_r:=\inf_{\sigma \in \Omega} \gamma_r(\sigma)$$ (in particular $\gamma_0$ is the multigamma admissibility).
Now, consider any update mark $(i,t,u) \in \PP$; if one uses the update function $\varphi$ of \eqref{eq:phi_Ising} and reveals the configuration at time $t-$ (denoted here $\sigma_{t-}$, without specifying any initial time and configuration) on $i+B_r^*$, and if moreover $u \leq p_r^-(\sigma_{t-})$ (resp. $u>1-p_r^+(\sigma_{t-})$), then this is sufficient to determine that the new spin at $i$ will be $-1$ (resp. $+1$);
it means that with probability (at least) $\gamma_r$, one is able to determine the new spin produced by the update mark without looking at the current configuration outside $i+B_r^*$.

Therefore, we would like to say that when $u$ belongs to some subset of $[0,1]$ of Lebesgue-measure $\gamma_r$, then, whatever the current configuration, the new spin at $i$ can be computed only by revealing the configuration on $i+B^*_r$.
However, with $\varphi$ as update function, this will be true only when $$\ds u \in \bigcap_{\sigma \in \Omega} \left(p_r^-(\sigma) ;p_r^+(\sigma)\right]^c, $$ which is of lower measure: in fact, $\gamma_r > \inf_{\sigma \in \Omega} p_r^-(\sigma) + \inf_{\sigma \in \Omega} p_r^+(\sigma)$.
Therefore, we modify our update function to obtain the targeted property.
It still has two variables, but the first one is now a configuration and not just a real number representing $\sum J(x) \sigma(x)$.
The new update function $\Phi: \Omega \times [0,1] \to \{\pm 1\}$ is defined by
\begin{equation}\label{eq:phi_Ising_long_range}
    \Phi(\sigma,u) = \begin{cases}-1&\ds \si u \in [0,p_0^-] \cup \bigcup_{r\geq0} \left(\gamma_r(\sigma) , p^+_r(\sigma) + p_{r+1}^-(\sigma) \right]\\
    +1&\ds \si u \in (p_0^-,\gamma_0] \cup \bigcup_{r\geq0} \left(p^+_r(\sigma) + p_{r+1}^-(\sigma)  , \gamma_{r+1}(\sigma)\right]\end{cases}.
\end{equation}
This is well defined since $(p^-_r(\sigma))_{r\geq0}$ and $(p^+_r(\sigma))_{r\geq0}$ are non-decreasing and $\gamma_r(\sigma) \xrightarrow[r \to \infty]{}1$.
What has been done in \eqref{eq:phi_Ising_long_range} is partitioning $[0,1]$ into sub-intervals of length $p_0^-$, $p_0^+$, $p_1^-(\sigma) - p_0^-$, $p_1^+(\sigma)-p_0^+$, $p_2^-(\sigma)-p_1^-(\sigma)$ etc. and alternate the sign to determine the output of $\Phi(\sigma, u)$.
Hence, $\Phi(\sigma, u)$ equals $-1$ (resp. $+1$) for values of $u$ belonging to a subset of $[0,1]$ of Lebesgue measure $\lim_{r \to \infty} p_r^-(\sigma)$ (resp. $\lim_{r \to \infty} p_r^+(\sigma)$).
This construction guarantees two crucial properties for the update function:
\begin{itemize}
    \item if $U$ is a uniform random variable on $[0,1]$, then $\Phi(\sigma, U)$ is distributed according to the conditional law of the spin at the origin given that the rest of the configuration coincides with $\sigma$;
    \item if $u \leq \gamma_r$, since $p^-_r(\sigma)$, $p_r^+(\sigma)$ and $\gamma_r(\sigma)$ depend on $\sigma$ only through its restriction to $B_r^*$, one can determine $\Phi(\sigma,u)$ by revealing $\sigma$ only on $B_r^*$.
\end{itemize}

The updates in the coupled processes $(\sigma_t^{\omega, \tau})_{t \geq \tau}$, for all $\tau \in \R$ and $\omega \in \Omega$, defined in \eqref{eq:coupled_processes}, are now performed with this new update function $\Phi$. During an update of site $i$, the current configuration is viewed with $i$ as origin, i.e. we apply $\Phi$ to a translated-version of the configuration.
More precisely, for all $(i,t,u) \in \PP$ with $t\geq \tau$, we have
$$\forall j \neq i, \, \sigma^{\omega, \tau}_t(j)=\sigma^{\omega, \tau}_{t-}(j) \et \sigma^{\omega, \tau}_t(j)=\Phi(\theta_i\cdot \sigma^{\omega, \tau}_{t-},u)$$
where for all $\sigma \in \Omega$, $(\theta_i\cdot\sigma)(x):=\sigma(i+x)$.

\subsubsection{Backward exploration dominated by a Galton--Watson process.}
We explore the dynamics backwards as in Section \ref{sec:Backward exploration}.
We consider as before the last update of site $0$ before time $0$, generated by the update mark $(0,t_0,u_0) \in \PP$.
Let $r$ be the smallest integer such that $u_0 \leq \gamma_r$. Note that $r$ is random, but it is entirely determined by $u_0$.
The new value of the spin produced by this update mark --- which is the value of the spin at site $0$ and time $0$ in this realisation of the dynamics --- is therefore a measurable function of $u_0$ and of the restriction of the configuration at time $t_0-$ to the ball $B_r^*$.
We then continue the exploration starting from all the sites of this random ball. When $r=0$, the exploration ends.

Then, similarly to the finite-range case, it follows that the number of update marks discovered during the backward exploration is dominated by the total population of a Galton--Watson process of reproduction law
\begin{equation*}
    \pi:= \gamma_0 \delta_0 + \sum_{r \geq0} (\gamma_{r+1}-\gamma_r) \delta_{|B_{r+1}^*|}.
\end{equation*}
For $\xi$ a random variable of law $\pi$, one has
$$ \E[\xi]= \sum_{r \geq 0} (1-\gamma_r) |\partial B_{r+1}|.$$
If this expectation is less than 1, the Galton--Watson process dies out almost surely, and it yields a coding of $\mu_{\beta, J}$.
We prove the following identity, which allows us to estimate the decay of $1-\gamma_r$:
\begin{lemma}\label{lemma:tail_gamma_r}
    For all $r \geq 0$, one has
    \begin{equation*}
        1-\gamma_r =\tanh\left( \beta \sum_{|x|>r} |J(x)|\right).
    \end{equation*}
\end{lemma}

\begin{proof}
The proof is a simple computation using hyperbolic formulas.
Let us first compute $\gamma_r(\sigma)$ for any $\sigma \in \Omega$.
By definition of $p_r^-(\sigma)$ and $p_r^+(\sigma)$, summing both terms and using the fact\footnote{$(1+e^{a+b})^{-1} +(1+e^{-a+b})^{-1} = \frac{2+e^{-a+b}+e^{a+b}}{1+e^{a+b}+e^{-a+b}+e^{2b}} = \frac{e^a+a^{-a}+2e^{-b}}{e^a+e^{-a}+e^b+e^{-b}} = 1- \frac{e^b-e^{-b}}{e^a+e^{-a}+e^b+e^{-b}}.$} that for all $a,b \in \R$, $(1+e^{a+b})^{-1} +(1+e^{-a+b})^{-1} = 1- \frac{\sinh(b)}{\cosh(a)+\cosh(b)}$ gives
\begin{equation*}
    \gamma_r(\sigma) = 1-\dfrac{\sinh\left(2\beta\sum_{|x|> r} |J(x)| \right) }{\cosh\left(2\beta\sum_{|x|\leq r} J(x) \sigma(x)\right) + \cosh \left(2\beta\sum_{|x|> r} |J(x)| \right)}.
\end{equation*}
The minimum of this quantity over $\sigma \in \Omega$ is obtained when $\sum_{|x|\leq r} J(x) \sigma(x)$ is minimal, in absolute value.
Since $J(x)=J(-x)$, it is actually possible to make it null, by taking any configuration $\sigma$ satisfying $\sigma(-x)=-\sigma(x)$ for all $x \neq 0$.
Therefore, we obtain
$$\gamma_r = 1- \dfrac{\sinh\left(2\beta\sum_{|x|> r} |J(x)| \right)}{1+\cosh\left(2\beta\sum_{|x|> r} |J(x)| \right)} = 1- \tanh\left( \beta \sum_{|x|>r} |J(x)|\right)$$
using the fact\footnote{$(1+\cosh(2a))\tanh(a) = 2 \cosh(a)^2 \tanh(a) = 2 \cosh(a) \sinh(a) = \sinh(2a)$.} that for all $a\in \R$, $\sinh(2a) = (1+\cosh(2a))\tanh(a)$.
\end{proof}

Since $\tanh(t) \sim t$ as $t\to0$,
it follows that $\pi$ is integrable if and only if
$$\sum_{r \geq 0} |\partial B_{r+1}| \sum_{|x|>r} |J(x)| < \infty.$$
By Fubini--Tonelli theorem and equation \eqref{eq:asymptotic_size}, we find that this last condition is equivalent to
\begin{equation}\label{eq:integrability_condition}
    \sum_{x \neq 0} |x|^d |J(x)|<\infty,
\end{equation}
which holds when $\alpha>2d$ for $J$ satisfying \eqref{eq:form_J}.
Then, taking $\beta$ small enough ensures that $\E[\xi] <1$, so we have a coding for $\mu_{\beta, J}$.

\medskip
It remains to study the tail of the coding radius, which we denote by $R$. Recall that the coding volume $|B_R|$ is dominated by the total population of the Galton--Watson process.
Consider $(\xi_{n,i})_{n, i\in \N}$ independent random variables of law $\pi$. A realisation of the Galton--Watson process of reproduction law $\pi$ is classically defined by the following: let $Z_0=1$ and recursively for all $n \geq 0$, $Z_{n+1}=\sum_{i=0}^{Z_n-1} \xi_{n,i}$. Since $\E[\xi]<1$, one has almost surely $Z_n=0$ for $n$ large enough, and one can naturally associate a finite rooted tree with $(Z_n)_{n \geq 0}$ (the root having $\xi_{0,0}$ children, each of them having respectively $\xi_{1,0}, \xi_{1,1}, \xi_{1,2}, \dots$ children etc.). This tree has total size $T:=\sum_{n \geq 0} Z_n$, and we have
\begin{equation}\label{eq:domination}
    \forall n \in \N, \quad \P(|B_R|>n) \leq \P(T >n).
\end{equation}
Enumerate by $(\xi_k)_{1 \leq k \leq T}$ the number of children of the vertices of this tree, taken in the lexicographical order.
Let $(V_n)_{n \geq 0}$ be the associated \L ukasiewicz path, defined by $V_0=0$ and for all $1\leq n \leq T$, $V_n=V_{n-1}+\xi_n-1$. It follows that with probability 1, $V_n \geq 0$ for all $n< T$ and $V_T=-1$.
Therefore, one has $$\forall n \in \N, \quad \P(T>n)=\P(V_0, \dots, V_n \geq 0).$$
Thus, $T$ has the same distribution as the first time a random walk started at 0 with i.i.d. increments distributed as $\xi-1$ reaches $-1$.
We aim to apply \cite[Theorem 8.2.4]{Borovkov} about the tail of the hitting time of $-1$, which says in our framework that
\begin{equation*}
    \P(T >n) \underset{n \to \infty}{\sim} \E[T] \pi([1-m_1 n,\infty)) \text{ where } m_1=\E[\xi-1] <0,
\end{equation*}
but requires the increment distribution to have regularly varying tails.
Without further assumption on $J$, this is not the case of $\pi$; however, $\pi$ can be stochastically dominated by such a probability measure.
This is encapsulated in the following lemma, whose proof is postponed at the end of the paper.

\begin{lemma}\label{lemma:regular_variation}
    For $\beta$ small enough, there is a distribution $\hat \pi$ with regularly varying tails that stochastically dominates $\pi$: 
    $$\exists C'>0,\ \forall\beta<1/C',\ \forall n\geq 1, \quad\pi([n,\infty)) \leq C' \beta  n^{1-\alpha/d}=:\hat \pi ([n,\infty)).$$
\end{lemma}
N.B. The above inequality is in fact true for all $\beta$, but we need to take $\beta <1/C'$ so that $n \mapsto C'\beta n^{1-\alpha/d}$ defines an appropriate tail function of a probability measure on $\N$ (see the proof of the lemma for more details).

\subsubsection{Reproduction law with regularly varying tails}
We now consider a Galton--Watson process $(\hat Z_n)_{n \geq0}$ of reproduction law $\hat \pi$, coupled with $(Z_n)_{n\geq0}$ in such a way that $Z_n \leq \hat Z_n$ for all $n$ with probability 1 (this is possible to do so by considering $(\hat \xi_{n,i})_{i,n \in \N}$ i.i.d. of law $\hat \pi$ such that a.s. $\xi_{i,n} \leq \hat \xi_{i,n}$, and constructing $\hat Z_n$ the obvious way).
We choose $\beta$ sufficiently small to make $(\hat Z_n)_{n \geq 0}$ subcritical, that is
$$\beta < \left(C'\sum_{n \geq 1} n^{1-\alpha/d}\right)^{-1}.$$
Then, a.s. $T \leq \hat T:=\sum_{n\geq0} \hat Z_n$, and we can apply Theorem 8.2.4 in \cite{Borovkov} to $\hat T$, since $\hat \pi$ has regularly varying tails.\footnote{More precisely, to apply the statement, we use the fact that the shifted version of $\hat \pi$ obtained by subtracting 1 to the outcome, which is the jump distribution of the random walk, has regularly varying tails.}
It induces
\begin{equation}\label{eq:random_walk_tail}
    \P(T >n) \leq \P(\hat T>n) \underset{n \to \infty}{\sim} \E[\hat T] \hat \pi([1-m_1 n, \infty)),
\end{equation}
where $m_1=C'\beta \sum_{n \geq 1} n^{1-\alpha/d}-1<0$ is the mean of the reproduction law $\hat \pi$ minus 1.
Then, by the Doob's convergence theorem for positive supermartingales applied to $(V_{n \wedge \hat T})_{n \geq 0}$, and because $\hat T$ is finite almost surely, we obtain
$$-1= \lim_{n \to \infty }\E[V_{n \wedge \hat T}]=\lim_{n \to \infty} m_1 \E[\hat T \wedge n],$$ so $\E[\hat T]=-1/m_1.$

Therefore, combining \eqref{eq:domination}, \eqref{eq:random_walk_tail} and Lemma \ref{lemma:regular_variation}, we deduce the existence of a constant $C>0$, depending only on $\alpha$, $\beta$ and $d$, such that
$$\forall n \geq 1, \quad \P(|B_R|>n) \leq Cn^{1-\alpha/d}.$$
This concludes the proof, according to the estimates \eqref{eq:asymptotic_size} on the asymptotic size of a ball in $\Z^d$.

\begin{proof}[Proof of Lemma \ref{lemma:regular_variation}]
    Let $N_r:=|B_r^*|$ and define $r(n):=\max\{r:N_r\leq n\}$, so that $r(n)$ is the unique integer such that $N_{r(n)} \leq n < N_{r(n)+1}$. Telescoping the probabilities $\gamma_{r+1}-\gamma_r$ over $r\geq r(n)$ gives, for every $n\geq1$,
    $$\pi([n,\infty))=1-\gamma_{r(n)}.$$
    Hence, by Lemma \ref{lemma:tail_gamma_r} and the inequality $\tanh(t) \leq t$ for $t\geq0$,
    $$\pi([n,\infty)) \leq \beta \sum_{|x|>r(n)}|J(x)|.$$
    The assumption \eqref{eq:form_J} on $J$ implies that $\sum_{|x|>r(n)} |J(x)|$ is less than a constant times $r(n)^{d-\alpha}$;
    combined with \eqref{eq:asymptotic_size}, which yields $r(n) \sim v_d^{-1/d} n^{1/d}$, we obtain
    \begin{equation*}
        \pi([n,\infty)) \leq C' \beta n^{(d-\alpha)/d}
    \end{equation*}
    for some constant $C'>0$ that does not depend on $\beta$, which is the claimed inequality.
    To deduce that $\pi$ is dominated by a distribution with regularly varying tails, we first choose $\beta<1/C'$, and set $\hat \pi(\{0\}) :=1-C'\beta$. Then, set for all $n\geq 1$, $\hat \pi(\{n\}):=C'\beta(n^{1-\alpha/d}-(n+1)^{1-\alpha/d})$, so that we exactly have $\hat \pi([n,\infty))=C'\beta n^{1-\alpha/d}$. Clearly, $\hat \pi$ has regularly varying tails, with index $1-\alpha/d$.
    Moreover, we have $\pi([n,\infty)) \leq \hat \pi([n,\infty))$, i.e. $\hat \pi$ stochastically dominates $\pi$.
\end{proof}

\begin{remark}\label{rem:hyp_J}
    The property \eqref{eq:form_J} on the decay of $J$ has been used in two places in the proof: to ensure that \eqref{eq:integrability_condition} holds, so that $\pi$ is integrable, and in the above proof of Lemma \ref{lemma:regular_variation} to upper bound $\pi([n, \infty))$.
    If we replace this assumption on $J$ by
    $$|J(x)| \leq |x|^{-\alpha} L(|x|)$$
    for some slowly varying function $L$, we obtain in this last proof
    $$\forall n \geq1, \quad \pi([n,\infty)) \leq \beta \sum_{|x|>r(n)} |x|^{-\alpha} L(|x|).$$
    According to \eqref{eq:asymptotic_size} and by Karamata's theorem (see \cite[Theorem 1.5.11]{Bingham_Goldie_Teugels_1987}) for regularly varying sequences, one has
    $$\beta \sum_{|x|>r(n)} |x|^{-\alpha} L(|x|) \sim \beta \dfrac{\kappa_d}{\alpha-d}r(n)^{d-\alpha}L(r(n)) \quad \text{ as } n\to \infty.$$
    Since $r(n) \sim v_d^{-1/d} n^{1/d}$, the uniform convergence theorem for slowly varying functions \cite[Theorem 1.2.1]{Bingham_Goldie_Teugels_1987} implies that $L(r(n)) \sim L(n^{1/d})$, and we obtain
    $$\beta\sum_{|x|>r(n)} |x|^{-\alpha} L(|x|) \sim \beta\dfrac{\kappa_d}{\alpha-d} v_d^{(\alpha-d)/d} n^{(d-\alpha)/d} L(n^{1/d}) \quad \text{ as } n\to \infty.$$
    Applying Karamata's theorem again to the sequence in the RHS gives that $\pi$ is integrable if $\alpha>2d$, or even if $\alpha=2d$ and $L(n)=\ln(n)^{-(1+\varepsilon)}$ for example.
    Then, taking $\beta$ small enough, one can define a distribution $\hat \pi$ with regularly varying tails that stochastically dominates $\pi$, satisfying
    $$\forall n \geq 1, \quad \hat \pi([n,\infty))= \sup_{k \geq n} C'\beta k^{1-\alpha/d} L(k^{1/d}) \underset{n \to \infty}{\sim}C'\beta n^{1-\alpha/d} L(n^{1/d}) $$
    for some constant $C'$ (N.B. the above asymptotic equivalent, which justifies that $\hat \pi$ has regularly varying tails of index $1-\alpha/d$, is an immediate consequence of \cite[Theorem 1.5.3]{Bingham_Goldie_Teugels_1987}).
    One can then conclude as in the proof of Theorem \ref{thm:main} to obtain
    $$\forall n \geq 1, \quad \P(|B_R| >n) \leq C n^{1-\alpha/d} L(n^{1/d}).$$
\end{remark}

\subsection*{Acknowledgments}
I thank Vincent Beffara for our fruitful discussions during the early stages of this project, and also Loren Coquille, Paul Dario and Arnaud Le Ny for valuable discussions about long-range models.
Finally, I'm grateful to Jean-René Chazottes for encouraging me to write this paper.

\printbibliography
\end{document}